\documentclass[reqno,11pt]{amsart}

\usepackage{stix2}
\usepackage{amsmath,amsthm,amscd,amssymb,latexsym,upref,stmaryrd,enumitem}
\usepackage{graphicx, tikz}
\usetikzlibrary{calc}

\newtheorem{theorem}{Theorem}
\newtheorem{lemma}[theorem]{Lemma}
\newtheorem{corollary}[theorem]{Corollary}

\newtheorem*{mthm*}{Main Theorem}
\newtheorem{problem}[theorem]{Problem}

\newtheorem*{theorema*}{Theorem A}
\newtheorem*{problemb*}{Problem B}
\newtheorem*{problembstar*}{Problem B$^\prime$}  
\newtheorem*{conjecturec*}{Conjecture C}
\newtheorem*{conjecturecstar*}{Conjecture C$^\prime$}
\newtheorem*{theoremd*}{Theorem D}
\newtheorem*{theoreme*}{Theorem E}

\newcommand{\C}{\mathbb C}
\newcommand{\D}{\mathbb D}
\newcommand{\E}{\mathcal E}
\newcommand{\N}{\mathbb N}
\newcommand{\R}{\mathbb R}
\newcommand{\Z}{\mathbb Z}
\newcommand{\Q}{\mathbb Q}

\begin{document}

\title[Deep zero problems and the HRT conjecture]
{Deep zero problems and the\\ HRT conjecture}

\author{Yufei Li}
\address{School of Mathematics and Statistics, Northeast Normal University, 
Changchun, Jilin 130024, China}
\email{liyf495@nenu.edu.cn}

\author{Zeguang Liu}
\address{School of Mathematics and Statistics, Northeast Normal University, 
Changchun, Jilin 130024, China}
\email{liuzeguang205@nenu.edu.cn}

\author{Kehe Zhu}
\address{Department of Mathematics and Statistics, SUNY, Albany, NY 12222, USA}
\email{kzhu@albany.edu}



\date{\today}

\thanks{Liu is the corresponding author.}
\thanks{Research is partially supported by the National Natural Science Foundation 
of China, grant numbers 12571136, 12001089, and 12271328.}

\begin{abstract}
We investigate a ``deep zero problem'' proposed by Hedenmalm. We show that
there is a natural connection between Hedenmalm's problem and the classical HRT
conjecture in time-frequency analysis. This connection allows us to show that 
Hedenmalm's problem 5.2 in \cite{Hed2025} as well as some of its natural analogs
have affirmative answers.  
\end{abstract}

\maketitle

\section{Introduction}

Let $\Omega$ be a domain (connected, open, and nonempty set) in the complex
plane $\C$ and let $\N_0=\{0,1,2,3,\cdots\}$ denote the set of all nonnegative
integers. The classical {\it identity theorem} states that if $f$ is analytic on
$\Omega$ and if either $f^{(j)}(a)=0$ for some point $a\in\Omega$ and all
$j\in\N_0$ or $f(a_j)=0$ for a sequence $\{a_j\}$ in $\Omega$
that has an accumulation point in $\Omega$, then $f$ is identically $0$ on
$\Omega$.

There are two natural directions in which we can attempt to extend these identity 
theorems (or uniqueness theorems). The first concerns the condition that 
$f(a_j)=0$ for a sequence $\{a_j\}$ in $\Omega$, where $\{a_j\}$ does not 
necessarily have an accumulation point in $\Omega$. There can be additional 
assumptions on $f$ and $\{a_j\}$ which would imply that $f$ is identically $0$. 
The most prominent example is the following: if $f\in H^\infty$, the space of all 
bounded analytic functions on the open unit disc $\D$ in $\C$, and if 
$\sum(1-|a_j|)=\infty$, then the condition that $f(a_j)=0$ for all $j$ (counting
multiplicity) implies $f=0$. More generally, for any given space $X$ of analytic 
functions on $\Omega$, we can seek conditions on $\{a_j\}$ such that for any 
$f\in X$ the condition $f(a_j)=0$ for all $j$ (counting multiplicity) implies $f=0$. 
This problem of characterizing zero sequences for $X$ has only been solved for 
a very small number of spaces. The problem remains open for many classical 
function spaces, including the Bergman space of $\D$. 

The second direction concerns the condition $f^{(j)}(a)=0$  for all $j\in\N_0$, 
where $a$ is some fixed point in $\Omega$. Again, we can require $f$ 
to belong to a certain analytic function space $X$. The following problem has 
been studied in the literature. Let $a$ and $b$ be two points in 
$\Omega$, let $\E_0\subset\N_0$, and let $\E_1=\N_0\setminus\E_0$. Find 
additional assumptions (on $f$, $\E_0$, $a$, and $b$) such that the condition
$$\begin{cases} f^{(j)}(a)=0, &j\in \E_0,\\ f^{(j)}(b)=0, &j\in \E_1,\end{cases}$$
implies that $f=0$. See \cite{Abel1839, MAE1954, AI1947,JAK1968,JMW1964} 
for example.

More generally, if $d\ge2$ and if
\begin{itemize}
\item[(i)] $T_k$, $0\le k\le d-1$, are bounded linear operators on $X$, 
\item[(ii)] $\E_k$, $0\le k\le d-1$, are disjoint subsets of $\N_0$,
\item[(iii)] $a_k$, $0\le k\le d-1$, are points in $\Omega$,
\end{itemize}
we may ask to find additional assumptions under which the condition
$$\begin{cases}(T_0f)^{(j)}(a_0)=0,&j\in \E_0,\\
(T_1f)^{(j)}(a_1)=0,&j\in \E_1,\\
\vdots &\vdots\\
(T_{d-1}f)^{(j)}(a_{d-1})=0,&j\in \E_{d-1},\end{cases}$$
implies that $f=0$. Following Hedenmalm \cite{Hed2025}, we will call all 
these ``deep zero problems''. 

Of particular interest is the case when
$$\E_0\cup \E_2\cup\cdots\cup \E_{d-1}=\N_0.$$
For example,
$$\E_k=\{k+dj: j\in\N_0\},\quad 0\le k\le d-1,$$
give such a decomposition of $\N_0$. When $d=2$, $\E_0$ is the set of all
even integers in $\N_0$, while $\E_1$ is the set of all odd integers in $\N_0$.
If $\Omega$ is ``sufficiently symmetric'' and the space $X$ is ``sufficiently 
regular'', we may be able to move points around in $\Omega$ 
(translations for example) and so can adjust the operators involved and 
assume that  $a_0=a_1=\cdots=a_{d-1}$.

Hedenmalm's paper \cite{Hed2025}, as well as ours here, focuses on the case 
where $X$ is the Fock space $F^2$ of the complex plane and the operators $T$ 
on $F^2$ are Weyl unitary operators (weighted translation operators). Recall that 
$F^2$ consists of all entire functions $f$ on $\C$ such that
\begin{align*}
\|f\|^2=\int_{\C}|f(z)|^{2}\,\mathrm{d}\mu_{G}(z)<+\infty,
\end{align*}
where $\mathrm{d}\mu_G(z)=e^{-|z|^{2}}\mathrm{d}A(z)$ is the Gaussian
measure on $\C$. Here $\mathrm{d}A(z)=\pi^{-1}\mathrm{d}x\mathrm{d}y$, 
with $z = x + iy$, is the normalized area measure. For any point $\alpha\in\C$, 
the Weyl operator $U_\alpha: F^2\to F^2$ is defined by
$$U_\alpha f(z)=e^{-\frac12|\alpha|^2-\overline\alpha z}f(z+\alpha).$$
It is well known and easy to check that $U_\alpha$ is a unitary operator with
$U_\alpha^{-1}=U_\alpha^*=U_{-\alpha}$. For background information 
on the Fock space and Weyl unitary operators, see \cite{Zhu2012} for example.

One of the results proved in \cite{Hed2025} is the following.

\begin{theorema*}
Suppose $f\in F^2$ and $\E$ is the set of even integers in $\N_0$ 
(or the set of odd integers in $\N_0$). If for some $\beta\in\C$ we 
have $f^{(j)}(0)=0$ for all $j\in\E$ and 
$(U_\beta f)^{(j)}(0)=0$ for all $j\in\E^c$, then $f=0$.
\end{theorema*}

Similar results were also proved in \cite{MR2023} for Hardy, Bergman, 
and Dirichlet spaces of the unit disc.

In the same paper \cite{Hed2025}, Hedenmalm also raised the following
question, phrased in terms of the general setup above.

\begin{problemb*} 
Let $f\in F^2$ and let $\E_k=\{k+4j: j\in\N_0\}$ for $0\le k\le 3$. If there
exists a point $\beta\in\C$ such that
\begin{equation}
\begin{cases}
f^{(j)}(0)=0, &j\in\E_0,\\
(U_\beta f)^{(j)}(0)=0, &j\in\E_1,\\
(U_\beta f)^{(j)}(0)=0, &j\in\E_2,\\
(U_\beta f)^{(j)}(0)=0, &j\in\E_3,\end{cases}
\end{equation}
does it follow that $f=0$?
\end{problemb*}

Note that any entire function $f$ on $\C$ has a Taylor expansion
$$f(z)=\sum_{j=0}^\infty\frac{f^{(j)}(0)}{j!}\,z^j.$$
With the sets $\E_k$, $0\le k\le 3$, defined in Problem~B, it is easy 
to see that the condition $f^{(j)}(0)=0$ for all $j\in\E_0$ if and only if
$$f(z)+f(iz)+f(-z)+f(-iz)=0.$$
Similarly, the condition $f^{(j)}(0)=0$ for all $j\in\E_1\cup\E_2\cup\E_3$
if and only if $f(iz)=f(z)$. Thus Problem~B can be restated in terms of
functional equations as follows.

\begin{problembstar*}
Let $f\in F^2$. If there exists some $\beta\in\C$ such that
$$\begin{cases}
U_\beta f(iz)=U_\beta f(z),\\
f(z)+f(iz)+f(-z)+f(-iz)=0,
\end{cases}$$
does it follow that $f=0$?
\end{problembstar*}

The purpose of this note is to show that Problem B has a positive answer. 
In fact, we will prove the following.

\begin{mthm*}
Problem B as well as the corresponding problems for $d=3$ and $d=6$ 
all have affirmative answers.
\end{mthm*}

Our approach is to explore a certain connection between deep zero problems 
for the Fock space and the so-called HRT conjecture in time-frequency analysis. 
Although the HRT conjecture is still open, a known special case will enable us 
to prove the theorem above. It is natural for us to conjecture that the analog
of Problem B for any positive integer $d>1$ has an affirmative answer.

\section{The HRT Conjecture}\label{s2}

The following is a core open problem in time-frequency analysis, which 
is called the {\it Linear Independence Conjecture} or sometimes the 
{\it HRT Conjecture}. See \cite{Heil2007,HRT1996} for example. As usual,
we use $\R$ to denote the real number line.

\begin{conjecturec*}
If $g\in L^2(\R, \mathrm{d}x)\setminus\{0\}$ and $\{(a_k, b_k)\}_{k=1}^N$ 
are distinct points in $\R^2$, then the functions
\begin{align*}  
g_k(x)=e^{2\pi i b_kx}g(x-a_k),\quad 1\le k\le N,
\end{align*}
are linearly independent in $L^2(\R, \mathrm{d}x)$.
\end{conjecturec*}

We are actually going to use the Fock space version of the HRT conjecture. 
Recall that (see \cite{Zhu2012, Zhu2019} for example) the Bargmann 
transform is the integral operator defined by  
\begin{align*}
(\mathbf{B}g)(z)=\frac1{(2\pi)^{1/4}}\int_{\R}
e^{-izx+\frac{1}{2}z^2-\frac14x^2}g(x)\,\mathrm{d}x, \quad z\in\C.
\end{align*}
It is well known that $\mathbf{B}: L^2(\R, dx)\rightarrow F^2$ is a unitary 
transformation.  
 
For functions $g$ and $g_k$ in Conjecture C, it is well known (see 
\cite{Zhu2019} for example) and easy to check directly that
\[(\mathbf{B}g_k)(z)=e^{i\pi a_kb_k}(U_{\lambda_k}f)(z),\]
where 
$$f=\mathbf{B}g\quad{\rm and}\quad\lambda_{k}=-2\pi b_{k}-ia_{k}/2$$
for $1\le k\le N$. 

Thus the HRT conjecture can be stated in terms of 
the Fock space as follows.

\begin{conjecturecstar*}
If $\lambda_k$, $1\le k\le N$, are distinct points in $\C$ and 
$f\in F^2\setminus\{0\}$, then the functions $U_{\lambda_k}f$, 
$1\le k\le N$, are linearly independent in $F^2$.
\end{conjecturecstar*}

Although the HRT conjecture remains open, there has been a considerable 
amount of work in the literature concerning the problem. In the following 
theorem we collect a few partial results relevant to our work here. 

\begin{theoremd*}
The HRT conjecture holds true
\begin{enumerate}
\item[(a)] if $N \leq 3$,   
\item[(b)] or if $N = 4$ and, among the four points in $\{(a_k, b_k)\}_{k=1}^{4}$, 
two lie on one line and the other two lie on a second parallel line,
\item[(c)] or if
$$\left\{\begin{pmatrix}a_k\\ b_k\end{pmatrix}: 1\le k\le N\right\}
\subset A\Z^2+\begin{pmatrix} a\\ b\end{pmatrix},$$ 
where $\Z^2=\left\{\begin{pmatrix}n\\ m\end{pmatrix}: n,m\in\Z\right\}$ is the 
ordinary integer lattice in $\R^2$, $A$ is an invertible $2\times2$ matrix, 
and $(a,b)\in\R^2$.
\end{enumerate}
\end{theoremd*}

Part (a) follows from Theorem 1 in \cite{HRT1996}, part (b) follows from 
Theorem 1.4 of \cite{DZ2012}, and part (c) follows from Proposition 1.3 
of \cite{LPA1999}. See \cite{HRT1996, LDF2022, SDW2015} and references 
therein for further information and results about the HRT conjecture.

We will call the set $A\Z^2+\begin{pmatrix} a\\ b\end{pmatrix}$ a regular 
lattice in $\R^2$. Similarly, we can also talk about regular lattices in $\C$,
which are sets of the form $\Z e_1+\Z e_2+e$, where $e$ is any complex
number, and $e_1,e_2$ are any two $\R$-independent complex numbers. 

It is easy to see that the mapping
$$z=-2\pi y-ix/2$$ 
(see earlier remarks about the Bargmann transform) from the real 
Euclidean plane $\R^2$ to the complex plane $\C$ has the following properties:
\begin{enumerate}
\item[(i)] It maps straight lines in $\R^2$ to straight lines in $\C$.
\item[(ii)] It maps parallel lines in $\R^2$ to parellel lines in $\C$.
\item[(iii)] It maps any regular lattice in $\R^2$ to a regular lattice in $\C$.
\end{enumerate}
Therefore, a corresponding version of Theorem D also holds for the Fock
space version of the HRT conjecture when points in $\R^2$ are replaced
by points in $\C$.

\begin{theorem}\label{t1}
Let $d\in\{1,2,3,4,6\}$ and $\beta\in\C\setminus\{0\}$. Then (the Fock 
space version of) the HRT conjecture is valid for the points
$$\lambda_k=\beta e^{2k\pi i/d},\qquad 0\le k\le d-1.$$
\end{theorem}

\begin{proof}
The case $d=1$ is trivial. The cases $d=2$ and $3$ follow from part (a) 
or part (c) of Theorem D. It is clear that the case $d=4$ follows from 
part (b) or part (c) of Theorem D.

Less obvious is the case $d=6$, which we show also follows from 
part (c) of Theorem D. In fact, with
$$\omega=e^{2\pi i/6}=e^{\pi i/3},$$
we have $\lambda_k=\beta\omega^k$ for $0\le k\le 5$, and
$$\omega^0= 1,\ \omega^1=\omega,\ \omega^2=\omega-1,
\ \omega^3=-1,\ \omega^4=-\omega,\ \omega^5=1-\omega.$$
This shows that, for every $0\leq k\leq 5$, we have
$$-\beta\omega^k=-\beta(m_k+n_k\omega),
\qquad m_k, n_k\in\{-1,0,1\}.$$
It follows that the points in $\{1,\omega,\cdots,\omega^5\}$ belong to
the regular lattice $\Z+\omega\Z$ in $\C$ (see picture below). 
Consequently, the points in $\{\beta e^{2k\pi i/6}: 0\le k\le 5\}$ belong
to the regular lattice $\beta(\Z+\omega\Z)$. This completes the proof 
of the theorem.
\end{proof}

\bigskip
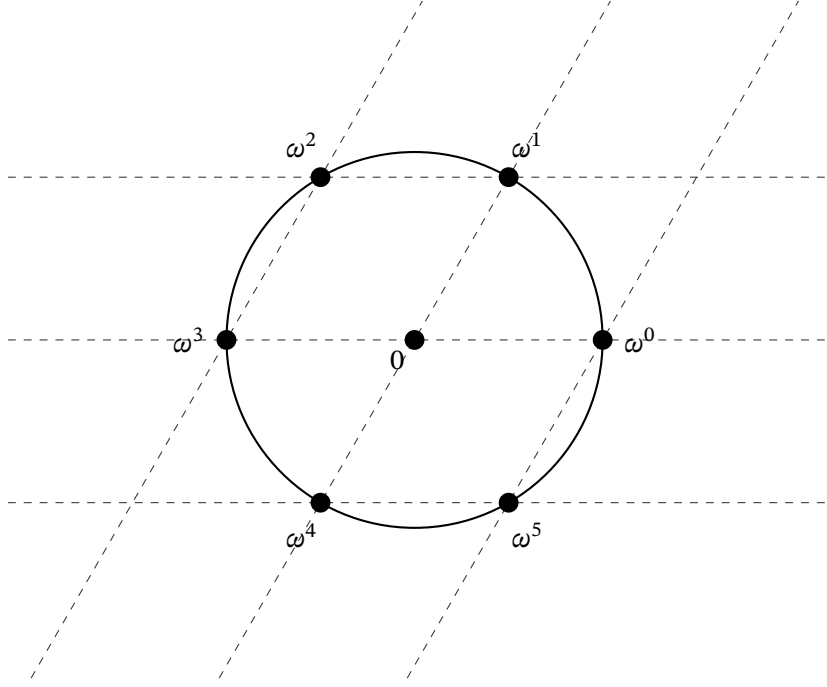
\begin{tikzpicture}[scale=2.5]
\clip (-2.2, -1.8) rectangle (2.2, 1.8);
\def\R{1} 
\draw[dashed, black!80] (-2.5, 0) -- (2.5, 0);
\draw[dashed, black!80] (-2.5, {sin(60)}) -- (2.5, {sin(60)});
\draw[dashed, black!80] (-2.5, {-sin(60)}) -- (2.5, {-sin(60)});
\draw[dashed, black!80] (240:2.5) -- (60:2.5);
\draw[dashed, black!80] ($(-1,0) + (240:2.5)$) -- ($(-1,0) + (60:2.5)$);
\draw[dashed, black!80] ($(1,0) + (240:2.5)$) -- ($(1,0) + (60:2.5)$);
\draw[thick, black] (0,0) circle (\R);
\fill[black] (0,0) circle (1.5pt) node[below left] {$0$};
\foreach \x/\angle in {0/0, 1/60, 2/120, 3/180, 4/240, 5/300} 
{\coordinate (w\x) at (\angle:\R);
\fill[black] (w\x) circle (1.5pt);
\node at (\angle:{\R+0.2}) {$\omega^{\x}$};}
\end{tikzpicture}
\bigskip

For any integer $d>1$ the points $\omega_k=\exp(2k\pi i/d)$,
$0\le k\le d-1$, are so symmetrically distributed on the unit circle.
It is very tempting for us to guess that they must always belong to 
a certain regular lattice. It turns out that this is NOT the case! 
We will need some algebraic preliminaries to prove this.

Let $\N$ denote the set of all positive integers. Recall that the 
Euler function $\phi: \N\to\N$ is defined by
$$\phi(n) = \begin{cases}
\prod_{j=1}^{r}p_{j}^{k_{j}-1}(p_{j}-1), & n>1,\\
1, & n=1,\end{cases}$$
where 
$$n=p_{1}^{k_{1}} p_{2}^{k_{2}} \cdots p_{r}^{k_{r}}, 
\qquad p_1<\cdots<p_r,$$
is the unique prime factorization for $n>1$. See \cite[pp. 95-96]{L2002}.

\begin{lemma}\label{l2}
$\phi(n)$ is even for all $n>2$. 
\end{lemma}

\begin{proof}
If $n$ contains an odd prime factor $p$, then $\phi(n)$ includes 
the factor $p-1$, which is even. If $n$ contains no odd prime factors, 
then $n=2^{k}$ with $k\geq 2$, and in this case $\phi(n)=2^{k-1}$ 
is also even.
\end{proof}

\begin{lemma}\label{l3}
The solutions to the equation $\phi(n)=2$ are $n=3, 4, 6$.
\end{lemma}

\begin{proof}
It is clear that $\phi(n)=2$ for $n=3, 4, 6$. On the other hand, if $n$
is any solution of $\phi(n)=2$, then $n>1$ and
$n=2^{k_{1}}3^{k_{2}}5^{k_{3}}\cdots$, where each $k_{j}\geq 0$ 
and not all are zero. If $n$ has a prime factor $p\geq 5$, then 
$\phi(n)\geq 4$. Therefore, $k_{j}=0$ for all $j\geq 3$ and and we
must have $n = 2^{k_{1}}3^{k_{2}}$.

\begin{enumerate}
\item[(a)] If $k_{1}=0$, then $2=\phi(n)=\phi(3^{k_2})=3^{k_2-1}(3-1)$,
which gives $k_{2}=1$ and $n=3$.
\item[(b)] If $k_{2}=0$, then $2=\phi(n)=\phi(2^{k_1})=2^{k_1-1}(2-1)$,
which gives $k_{1}=2$ and $n=4$.
\item[(c)] If $k_{1}$ and $k_{2}$ are both nonzero, then 
$2=\phi(n)=\phi(2^{k_{1}}3^{k_{2}})=2^{k_{1}-1}3^{k_{2}-1}\cdot 2$, 
which gives $k_{1}=k_{2}=1$ and $n=6$.
\end{enumerate}
This proves the desired result.
\end{proof}

Let $\mathbb{F}$ be a subfield of another field $\mathbb{E}$. Following
\cite[pp.223-224]{L2002}, we denote by $[\mathbb{E}:\mathbb{F}]$ the 
dimension of $\mathbb{E}$ as vector space over $\mathbb{F}$. 

\begin{theoreme*}\cite[Theorem 3.1, pp.278]{L2002}
Let $d$ be a positive integer and $\omega=e^{2\pi i/d}$. Then 
$$[\mathbb{Q}(\omega): \mathbb{Q}]=\phi(d),$$ 
where $\mathbb{Q}$ denotes the field of rational numbers, 
$$\mathbb{Q}(\omega)=\text{span}_{\mathbb{Q}}
\left\{\omega^k:k=0,\cdots,d-1\right\},$$ 
and $\phi$ is the Euler function.
\end{theoreme*}

We can now show that, for $4<d\not=6$, the $d$-th roots of unity
cannot be placed in any regular lattice. There is a chance that the 
following result is known to experts in algebra. But we are not
able to find a specific reference, so a complete proof is given here.

\begin{theorem}\label{t4}
Suppose $d\in\N$ and $\omega=e^{2\pi i/d}$. Then the
following are equivalent.
\begin{enumerate}
\item[(a)] The points
$$\lambda_k=\omega^k=e^{2k\pi i/d},\qquad 0\le k\le d-1,$$
belong to a regular lattice on the complex plane.
\item[(b)] $n\in\{1,2,3,4,6\}$.
\end{enumerate}
\end{theorem}

\begin{proof}
It is obvious that the points in $\{\omega^k: 0\le k\le d-1\}$ 
belong to a regular lattice in $\C$ when $d\in\{1,2,3,4\}$. 
By Theorem~\ref{t1}, this is also true when $d=6$.

Next we fix $d>4$ and assume that 
$\{\lambda_0,\lambda_1,\cdots,\lambda_{d-1}\}$ belong 
to a regular lattice
$$L=\Z e_1+\Z e_2+e$$ 
in the complex plane, where $e\in\C$ and $e_1,e_2$ are two 
complex numbers that are linearly independent with respect 
to $\R$.

For $1\le k\le d-1$ we write
$$v_k=\lambda_k-\lambda_0=n_{k1}e_1+n_{k2}e_2,$$
where $n_{k1}$ and $n_{k2}$ are integers. Then for $k\ge2$
we can write
$$\begin{pmatrix}v_1\\ v_k\end{pmatrix}=
\begin{pmatrix}n_{11} & n_{12}\\ n_{k1} & n_{k2}\end{pmatrix}
\begin{pmatrix}e_1\\ e_2\end{pmatrix}.$$
With respect to the real numbers, $v_1$ and $v_k$ are linearly
independent, and $e_1$ and $e_2$ are linearly independent. It
follows that each of the above $2\times2$ matrices of integers
is invertible. In particular, we have
$$\begin{pmatrix}e_1\\ e_2\end{pmatrix}=\begin{pmatrix}
n_{11} & n_{12}\\ n_{21} & n_{22}\end{pmatrix}^{-1}
\begin{pmatrix}v_1\\ v_2\end{pmatrix}.$$
Thus for $k\ge3$ we can write
$$\begin{pmatrix}v_1\\ v_k\end{pmatrix}=\begin{pmatrix}
n_{11} & n_{12}\\ n_{k1} & n_{k2}\end{pmatrix}
\begin{pmatrix}n_{11} & n_{12}\\ n_{21} & n_{22}\end{pmatrix}^{-1}
\begin{pmatrix}v_1 \\ v_2\end{pmatrix}.$$
Elementary linear algebra tells us that the inverse of an invertible
matrix with integer entries has rational entries. It follows that each
$v_k$, $k\ge3$, is a linear combination of $v_1$ and $v_2$ with
rational coefficients. Consequently, the sets 
$$\{v_0, v_1,\cdots, v_{d-1}\}\quad{\rm and}\quad
\{\lambda_0,\lambda_1,\cdots,\lambda_{d-1}\}$$ 
are both contained in ${\rm Span}_{\Q}\{1,\omega,\omega^2\}$,
which implies $\left[\Q(\omega) : \Q\right]\le3$. This along with
Lemma~\ref{l2} and Theorem E shows that $\phi(d)=2$. An
application of Lemma~\ref{l3} then gives $d=6$.
\end{proof}

Obviously, Theorem \ref{t4} remains true if the points $\lambda_k$
are replaced by $\lambda_k=\beta\omega^k$, $0\le k\le d-1$, where
$\beta$ is any nonzero complex number.

\section{A deep zero problem}

In this section we use the HRT conjecture to study a certain deep
zero problem for the Fock space. We begin with some preliminary
results about composition operators on the Fock space.

Given an integer $d >1$, we  again let $\omega = e^{2\pi i/d}$.
Consider the function $\varphi(z)=\omega z$. Obviously, $\varphi$
is just a rotation of the complex plane about the origin, so it induces 
a unitary composition operator $C_{\varphi}$ on the Fock space,
$$C_{\varphi}f(z) = f\circ\varphi(z)=f(\omega z),\qquad f\in F^2.$$

It is clear that any function
$$f(z)=\sum_{j=0}^{\infty}a_{j}z^{j}$$
in $F^2$ admits a unique decomposition
$$f(z)=\sum_{k=0}^{d-1}f_{k}(z),$$
where
$$f_{k}(z)=\sum_{j=0}^{\infty}a_{k+jd}z^{k+jd},\qquad 0\le k\le d-1.$$
For $0\le k\le d-1$ we define the projection operator $P_k: F^2\to F^2$ by 
$P_{k}f=f_{k}$. It is clear that the decomposition of $f$ above is orthogonal, 
that is, we actually have
$$I=\bigoplus_{k=0}^{d-1}P_k,$$
where $I$ is the identity operator on $F^2$. 

\begin{lemma}\label{l6}
With the notation above, we have $C_{\varphi}P_k=\omega^kP_k$ and
$$P_{k}f(z)=\frac1d\sum_{m=0}^{d-1}\omega^{-km}f(\omega^mz)
=\frac1d\sum_{m=0}^{d-1}\omega^{-km}C_{\varphi}^mf(z)$$
for $0\le k\le d-1$ and $f\in F^2$.
\end{lemma}

\begin{proof}
For any $f(z)=\sum_{n=0}^\infty a_nz^n$ in $F^2$ we have
\begin{align*}
C_\varphi P_kf(z)&=C_\varphi f_k(z)
=\sum_{j=0}^\infty a_{k+dj}(\omega z)^{k+dj}\\
&=\sum_{j=0}^\infty a_{k+dj}\omega^{k+dj}z^{k+dj}
=\omega^k\sum_{j=0}^\infty a_{k+dj}z^{k+dj}\\
&=\omega^kf_k(z)=\omega^kP_kf(z).
\end{align*}
Thus $C_\varphi P_k=\omega^k P_k$. It then follows that
\begin{align*}
\sum_{m=0}^{d-1}\omega^{-km}f(\omega^mz)&=
\sum_{m=0}^{d-1}\omega^{-km}C^m_\varphi f(z)
=\sum_{m=0}^{d-1}\omega^{-km}\sum_{n=0}^{d-1}
C^m_\varphi P_nf(z)\\
&=\sum_{m=0}^{d-1}\omega^{-km}\sum_{n=0}^{d-1}
\omega^{mn}P_nf(z)
=\sum_{n=0}^{d-1}P_nf(z)\sum_{m=0}^{d-1}(\omega^{n-k})^m\\
&=dP_kf(z)+\sum_{n=0, n\not=k}^{d-1}\frac{1-\omega^{d(n-k)}}
{1-\omega^{n-k}}P_nf(z)\\
&=dP_kf(z).
\end{align*}
This completes the proof of the lemma.
\end{proof}

The following lemma establishes the commutation relation between 
the Weyl operator $U_{\alpha}$ and the composition operator 
$C_{\varphi}$.

\begin{lemma}\label{l5}
Let $d>1$, $\omega=e^{2\pi i/d}$, $\varphi(z)=\omega z$, and 
$\beta\in\C$. Then
\begin{align*} 
C_{\varphi}U_{\beta}f=U_{\beta\omega^{-1}}C_{\varphi}f,\quad f\in F^2.
\end{align*}
\end{lemma}

\begin{proof}
For any $f\in F^2$ we have
\begin{align*}
(C_{\varphi}U_{\beta}f)(z)
&=e^{-\frac{1}{2}|\beta|^2-\bar{\beta}\omega z}f(\omega z+\beta) \\
&=e^{-\frac{1}{2}|\beta|^2-\bar{\beta}\omega z} 
f(\omega(z+\omega^{-1}\beta)) \\
&=e^{-\frac{1}{2}|\beta|^2-\bar{\beta} \omega z} 
(C_{\varphi}f)(z+\omega^{-1} \beta) \\
&=e^{-\frac{1}{2}|\beta\omega^{-1}|^2-\overline{\beta\omega^{-1}}z}
(C_{\varphi}f)(z+\omega^{-1} \beta) \\
&=(U_{\beta\omega^{-1}}C_{\varphi}f)(z).
\end{align*}
This completes the proof.
\end{proof}

The main result of this section is the following. As a consequence of it, 
we will obtain Hedenmalm's Theorem A as a special case, answer 
Hedenmalm's Problem B affirmatively, and derive a couple of additional 
new results.

\begin{theorem}\label{t6}
Let $d>1$ be an integer, $\beta_k=\beta\in\C$ for $k=1,\ldots,d-1$, and
$\E_k=\{k+dj: j\in\N_0\}$ for $0\le k\le d-1$. If Conjecture C$^\prime$ 
holds for $N=d$ and $\lambda_k=-\beta\omega^{-k+1}$, $0\le k\le d-1$,
where $\omega=e^{2\pi i/d}$, then the deep zero problem
\begin{equation}\label{eq2}
\begin{cases}
f^{(j)}(0)=0, &j\in\E_0,\\
(U_{\beta_1}f)^{(j)}(0)=0, &j\in\E_1,\\
\vdots & \vdots\\
(U_{\beta_{d-1}}f)^{(j)}(0)=0, &j\in\E_{d-1},\end{cases}
\end{equation}
for $f\in F^2$ has an affirmative answer, that is, the condition 
above implies $f=0$. 
\end{theorem}

\begin{proof}
Let $h=U_{\beta}f$. The conditions in (\ref{eq2}) are the same as
$$\begin{cases}
(U_{-\beta}h)^{(j)}(0)=(U_{\beta}^{-1}h)^{(j)}(0)=0, &j\in\E_{0},\\[8pt]
h^{(j)}(0)=0,  &j\in\E_k, 1\le k\le d-1,
\end{cases}$$
which imply that $P_{0}U_{-\beta}h = 0$ and $h_{k}=P_{k}h=0$ for
$1\le k \le d-1$. Write $h=\sum_{k=0}^{d-1}h_{k}$ and apply 
Lemmas~\ref{l6} and \ref{l5}. We obtain
\begin{align*}
0&=P_{0}U_{-\beta}h =P_{0}U_{-\beta}h_{0}\\ 
&=\frac{1}{d}\sum_{m=0}^{d-1}C_{\varphi}^{m}U_{-\beta}h_{0}
=\frac{1}{d}\sum_{m=0}^{d-1}U_{-\beta\omega^{-m}}
C_{\varphi}^{m}h_{0}\\
&=\frac{1}{d}\sum_{m=0}^{d-1}U_{-\beta\omega^{-m}}h_{0}
=\frac{1}{d}\sum_{m=0}^{d-1}U_{-\beta\omega^{-m}}h.
\end{align*}
If Conjecture C$^\prime$ holds for $N=d$ and the points $\lambda_k
=-\beta\omega^{-k}$, $0\le k\le N-1$, then we must have $h=0$,
and so $f=0$. This proves the desired result.
\end{proof}

The following result is now a consequence of Theorem~\ref{t1}, 
Theorem D, and Theorem \ref{t6}. The case $d=2$ yields 
Theorem A, the case $d=4$ gives an affirmative answer to 
Problem B, and the cases $d=3$ and $d=6$ are something 
that has not been considered before.

\begin{corollary}\label{c7}
For $2\le d\leq 4$ or $d=6$ let $\beta_{0}=0$ and 
$\beta_1=\cdots=\beta_{d-1}=\beta\in\C$. If $f\in F^2$ satisfies 
\begin{align*}
(U_{\beta_j} f)^{(j)}(0) = 0, ~~\textnormal{for each}~j 
\in \mathcal{E}_{k}~\textnormal{and}~k = 0,\ldots,d-1, 
\end{align*}
then $f=0$.
\end{corollary}

\begin{proof}
If $\beta=0$, then the assumptions imply that all Taylor coefficients of $f$ 
at the origin are $0$, so $f=0$. The case $\beta\neq 0$ follows from
Theorem~\ref{t1}, part (c) of Theorem D, and Theorem~\ref{t6}.  
\end{proof}

We conclude the paper with the following problem, a natural special case
of the HRT conjecture.

\begin{problem}\label{p8}
For any positive integer $d$ let $\lambda_k=e^{2k\pi i/d}$, $0\le k\le d-1$.
Is the HRT conjecture true for the points in $\{\lambda_k\}$, that is, if
$f\in F^2\setminus\{0\}$, are the functions in $\{U_{\lambda_k}f\}$ 
linearly independent?
\end{problem}

Although the full HRT conjecture remains open and seems to be a very
hard problem, we feel that this special case may not be that difficult. 
Its solution together with our Theorem~\ref{t6} will then yield a nice
result about deep zero problems. By the previous section, 
Problem~\ref{p8} for the case $d\in\{1,2,3,4,6\}$ has an affirmative 
answer, so the interesting remaining cases are $d=5$ and $d>6$, 
that is when the roots of unity no longer lie in a regular lattice but 
we suspect that the HRT conjecture may still be true.


\begin{thebibliography}{99}

\bibitem{Abel1839}
N. H. Abel, Sur les fonctions g\'{e}n\'{e}ratrices et leurs d\'{e}terminants, 
in {\sl{\OE}uvres compl\`{e}tes,}
Gr{\o}ndahl \& Son, Christiania, 1839, pp. 77--88.

\bibitem{DZ2012}
C. Demeter, and A. Zaharescu,
Proof of the HRT conjecture for (2,2) configurations,
{\sl J. Math. Anal. Appl.} \textbf{388} (2012), no. 1, 151--159.

\bibitem{MAE1954} M. A. Evgrafov, 
{\sl The Abel-Gon\v{c}arov interpolation problem,}
Gosudarstv. Izdat. Tehn.-Teor. Lit., Moscow, 1954, 126 pp.

\bibitem{AI1947}
A. O. Gel'fond and I. I. Ibragimov, 
On functions whose derivatives are zero at two points,
{\sl Izv. Akad. Nauk SSSR Ser. Mat.} \textbf{11} (1947), 547--560.


\bibitem{Hed2025}
H. Hedenmalm, Deep zero problems,
{\sl J. Anal. Math.} \textbf{156} (2025), no. 1, 83--95.

\bibitem{Heil2007}
C. Heil, History and evolution of the density theorem for Gabor frames,
{\sl J. Fourier Anal. Appl.} \textbf{13} (2007), 113-166.

\bibitem{HRT1996}
C. Heil, J. Ramanathan, and P. Topiwala, 
Linear independence of time-frequency translates,
{\sl Proc. Amer. Math. Soc.} \textbf{124} (1996), no. 9, 2787--2795.

\bibitem{JAK1968} Yu. A. Kaz'min,
Certain forms of power series expansions of entire functions of exponential type,
{\sl Vestnik Moskov. Univ. Ser. I Mat. Meh.} \textbf{23} (1968), no. 1, 49--63.

%\bibitem{KUZ2023}
%A. Kulikov, A. Ulanovskii, and I. Zlotnikov,
%Completeness of certain exponential systems and zeros of lacunary polynomials,
%{\sl Adv. Math.} \textbf{421} (2023), Paper No. 109016, 17 pp.

\bibitem{L2002}
S. Lang, {\sl Algebra} (revised 3rd edition), Springer, New York, 2002. 

\bibitem{LDF2022}
D. Li,  Linear independence of a finite set of time-frequency shifts,
{\sl Front. Math. China} \textbf{17} (2022), no. 4, 501--509.

\bibitem{LPA1999}
P. A. Linnell, 
von Neumann algebras and linear independence of translates,
{\sl Proc. Amer. Math. Soc.} \textbf{127} (1999), no. 11, 3269--3277.

\bibitem{MR2023}
A. Montes-Rodr\'{i}guez,
Zeros under unitary weighted composition operators in the Hardy and 
Bergman spaces, {\sl Mediterr. J. Math.} \textbf{20} (2023), no. 2, 
Paper No. 82, 9 pp.

\bibitem{SDW2015}
D. W. Stroock, Remarks on the HRT conjecture,
{\sl Lecture Notes in Mathematics} \textbf{2137}, 
Springer, Cham, 2015, 603--617.

\bibitem{JMW1964} J. M. Whittaker, 
{\sl Interpolatory function theory},
Stechert-Hafner, Inc., New York, 1964, v+107 pp.

\bibitem{Zhu2012}
K. Zhu,  {\sl Analysis on Fock spaces}, Springer, New York, 2012, x+344 pp.

\bibitem{Zhu2019} 
K. Zhu, Towards a dictionary for the Bargmann transform,
in {\sl Handbook of Analytic Operator Theory}, Chapman \& Hall/CRC, 2019,
319--349.

\end{thebibliography}
\end{document}